\documentclass[a4paper, 12pt]{amsart}
\usepackage{amssymb}
\usepackage{amsthm}
\usepackage{amsthm}
\usepackage{amsmath,amscd}
\usepackage[mathscr]{euscript}
\usepackage[all]{xy}
\usepackage{color}
\usepackage[dvipsnames]{xcolor}
\usepackage[utf8]{inputenc}
\normalfont
\usepackage[T1]{fontenc}
\usepackage[textwidth=14cm,hcentering]{geometry}
\usepackage[colorlinks=true,linkcolor=blue,citecolor=red]{hyperref}
\usepackage{enumerate}
\usepackage[shortlabels]{enumitem}
\usepackage{stackrel}
\usepackage{comment}
\usepackage{hyperref}
\usepackage{mdwlist}
\usepackage{color}
\usepackage[dvipsnames]{xcolor}
\usepackage{array} 
\input xy
\xyoption{all}

\usepackage{tikz}
\usepackage{tikz-cd}
\usetikzlibrary{arrows,calc,matrix,topaths,positioning,scopes,shapes,decorations}

\usepackage{enumitem}

\makeatletter
\@namedef{subjclassname@2020}{\textup{2020} Mathematics Subject Classification}
\makeatother

\newtheorem{prop}{Proposition}[section]
\newtheorem{teo}[prop]{Theorem}
\newtheorem{lem}[prop]{Lemma}
\newtheorem{cor}[prop]{Corollary}

\theoremstyle{definition}
\newtheorem{defi}[prop]{Definition}
\newtheorem{example}[prop]{Example}

\newtheorem{rmk}[prop]{Remark}
\newtheorem{notation}[prop]{Notation}

\numberwithin{prop}{subsection} 

\newcommand{\ZZ}{\mathbb{Z}}

\newcommand{\Cc}{\mathcal{C}}
\newcommand{\Dd}{\mathcal{D}}
\newcommand{\Ee}{\mathcal{E}}

\newcommand{\Ss}{\mathcal{S}}

\newcommand{\Ww}{\mathcal{W}}

\newcommand{\Sur}{\mathcal{S}ur}
\newcommand{\Fib}{\mathcal{F}ib}

\newcommand{\Tot}{\mathrm{Tot}}
\newcommand{\Hom}{\mathrm{Hom}}
\newcommand{\Id}{\mathrm{Id}}

\newcommand{\Ker}{\mathrm{Ker}}

\newcommand{\kk}{R}

\newcommand{\ncpx}{n\text{-}\mathrm{mC}_\kk} 
\newcommand{\mcpx}{\mathrm{mC}_\kk} 
\newcommand{\rbc}{r\text{-}\mathrm{bC}_\kk}            
\newcommand{\modu}{\mathrm{Mod}_\kk} 		
\newcommand{\fcpx}{\mathbf{F}\mathrm{C}_\kk} 			
\newcommand{\spse}{\mathbf{SpSe}_\kk} 			
\newcommand{\Tc}{{\mathrm{tC}_R}} 		

\newcommand{\simr}[1]{\begin{array}{c}\vspace{-.3cm} \simeq\\ \vspace{-.4cm}\text{\tiny{$#1$}}\vspace{.36cm} \end{array}}

\newcommand{\pb}{\ar@{}[dr]|{\mbox{\LARGE{$\lrcorner$}}}}
\newcommand{\lra}{\longrightarrow}

\title{Homotopy theory of spectral sequences}

\author{Muriel Livernet}
\address[M. Livernet]{
Univ Paris Cit\'e, Institut de Math\'ematiques de Jussieu-Paris Rive Gauche, CNRS, Sorbonne Universit\'e, 8 place Aur\'elie Nemours, F-75013, Paris, France}
\email{muriel.livernet@imj-prg.fr}

\author{Sarah Whitehouse}
\address[S. Whitehouse]{
School of Mathematics and Statistics\\ 
University of Sheffield\\ S3 7RH\\ England}
\email{s.whitehouse@sheffield.ac.uk }

\thanks{}



\begin{document}

\begin{abstract}Let $\kk$ be a commutative ring with unit. We
consider the homotopy theory of the category of spectral sequences of $\kk$-modules
with the class of weak equivalences given by those morphisms inducing
a quasi-isomorphism at a certain fixed page. We show that this admits a structure
close to that of a category of fibrant objects in the sense of Brown and in particular the structure of
a partial Brown category with fibrant objects. We use this to compare with related structures on
the categories of
multicomplexes and  filtered complexes.
\end{abstract}

\maketitle

\setcounter{tocdepth}{2}
\setcounter{secnumdepth}{2}
\tableofcontents

\section{Introduction}

Spectral sequences are important tools for computing homological and homotopical invariants. Many categories of interest have associated functorial
spectral sequences, generally via an associated filtered chain complex. 

The category of spectral sequences has a hierarchy of notions of weak equivalence. For $r\geq 0$, we have $E_r$-quasi-isomorphims, that is morphisms 
which are isomorphisms from the $r+1$ page onwards. In this paper we explore underlying homotopy theories with these weak equivalences.

Various categories with associated functorial spectral sequences, such as filtered complexes or multicomplexes, can be endowed with an 
$r$-model category structure, in which the weak equivalences are the maps inducing an isomorphisms from the $r+1$ page of the associated spectral sequence onwards~\cite{CELW20, FGLW}. This motivates a study of the corresponding homotopy theory in the category of spectral sequences itself.

After some preliminary definitions and discussion in Section~\ref{sec:prelim}, we introduce the category of spectral sequences
in Section~\ref{sec:specseq}. We study some basic properties, noting that this category
 is neither complete nor cocomplete. Therefore we cannot have model category structures
and we will work with a weaker setting for homotopy theory.

Many such settings, intermediate between a category with weak equivalences and a model category,
 have appeared in the literature. Examples include Waldhausen categories~\cite{W}, Cartan-Eilenberg categories~\cite{GNPR} and
categories of fibrant objects. The latter were introduced and studied by K.S.~Brown in~\cite{Brown}. A summary of this theory can be found in~\cite[I.9]{GJ}. That setting is the most relevant for us, but it is not precisely what we need.

In  Section~\ref{sec:abc}, we introduce the notion of an \emph{almost Brown category}. As the name suggests this is a structure closely related to Brown's notion of category of fibrant objects.
Like that setting, ours involves two distinguished classes of morphisms, weak equivalences and fibrations, satisfying certain axioms. We explore the connections as well as the relationship to the notion of partial Brown category in the sense of Horel~\cite{Horel}.

We show in Theorem~\ref{T:fund} that, for each $r\geq 0$, the category of spectral sequences admits the structure of 
an almost Brown category with $E_r$-quasi-isomorphims as weak equivalences and with fibrations characterised by surjectivity conditions. In particular, this means that we have a  
 partial Brown category with fibrant objects, in the sense of Horel~\cite{Horel}. Indeed, we have a version with functorial path objects. 

These results provide a context in which we can compare the homotopy theoretic structure of the category of spectral sequences with 
previous results establishing such structures for filtered complexes~\cite{CELW20} and for multicomplexes~\cite{FGLW}. 
We make a start on such comparisons in Section~\ref{sec:comparisons}.

\subsection*{Acknowledgements}
We would like to thank Geoffroy Horel for very helpful suggestions.

\section{Preliminaries}
\label{sec:prelim}
In this preliminary section, we collect the main definitions that we will use. We begin with bigraded modules and $r$-bigraded complexes
as these are underlying definitions for spectral sequences. Then we cover filtered complexes and multicomplexes, these being the
main categories to be compared with spectral sequences later on.

Throughout this paper, we let $\kk$ denote a commutative ring with unit. 

\subsection{Bigraded complexes}

In this section we let $r\geq 0$ be an integer.

\begin{defi}\label{D:bigraded modules} A \emph{bigraded $\kk$-module} $A$ is a collection of $\kk$-modules $A=\{A^{p,q}\}$ with $p,q\in\ZZ$. 
\end{defi}

\begin{defi}\label{D:r-complex} An \textit{$r$-bigraded complex}  is a bigraded $\kk$-module $A=\{A^{p,q}\}$ together with maps of $\kk$-modules
$\delta_r:A^{p,q}\to A^{p-r, q+1-r}$, called differentials, such that $\delta_r^2=0$. A \textit{morphism of $r$-bigraded complexes} is a map of bigraded modules commuting with the differentials.
\end{defi}

We denote by $\rbc$ the category of 
$r$-bigraded complexes. The homology of an $r$-bigraded complex  is a bigraded $\kk$-module and the category of $r$-bigraded modules has a natural class of quasi-isomorphisms, namely morphisms inducing isomorphisms on homology.

\subsection{Filtered complexes}
We consider unbounded complexes of $R$-modules
endowed with increasing filtrations indexed by the integers.

\begin{defi}\label{def:fm_obj}
A \textit{filtered $\kk$-module} $(A,F)$ is an $\kk$-module $A$ together with a family of submodules of A denoted $\{F_pA\}_{p\in\ZZ}$ 
indexed by the integers such that $F_{p-1}A\subseteq F_pA$ for all $p\in\ZZ$.
A \textit{morphism of filtered modules} is a morphism $f:A\to B$ of $\kk$-modules
which is compatible with filtrations:
$f(F_pA)\subseteq F_pB$ for all $p\in\ZZ$.
\end{defi}

\begin{defi}\label{def:fc_obj}A \textit{filtered complex} $(A,d,F)$ is an unbounded cochain complex $(A,d)$ together with
a filtration $F$ of each $\kk$-module $A^n$ such that 
$d(F_pA^n)\subseteq F_pA^{n+1}$ for all $p,n\in\ZZ$.
Note in particular that $(F_pA,d|_{F_p})$ is a subcomplex of $(A,d)$.
Denote by $\fcpx$ the category of filtered complexes of $R$-modules.
Its morphisms are given by morphisms of complexes compatible with filtrations.
\end{defi}

\begin{defi}\label{def:rhtpy_fcpx}
Let $f,g:A\to B$ be two morphisms of filtered complexes. 
An \textit{$r$-homotopy from $f$ to $g$} is 
a morphism $h:A\to B$ of graded  $\kk$-modules of degree $-1$, such that $dh+hd=g-f$ and $h(F_pA)\subseteq F_{p+r}B$ for all $p$.
We write $h:f\simr{r}g$.
\end{defi}

Every filtered complex $A$ has an associated spectral sequence $\{E_r(A),\delta_r\}_{r\geq 0}$.
The $r$-page $E_r(A)$ is an $r$-bigraded complex and may be written as the quotient
\[E_r^{p,q}(A)\cong Z_r^{p,q}(A)/B_r^{p,q}(A),\] where the \textit{$r$-cycles} are given by 
\[Z_r^{p,n+p}(A):=F_pA^{n}\cap d^{-1}(F_{p-r}A^{n+1})\]
and the \textit{$r$-boundaries} are given by $B_0^{p,n+p}(A)=Z_0^{p-1,n+p-1}(A)$ and
\[B_r^{p,n+p}(A):=Z_{r-1}^{p-1,n+p-1}(A)+ dZ_{r-1}^{p+r-1,n+p+r-2}(A)\text{ for }r\geq 1.\]
Given an element $a\in Z_r(A)$, we denote by $[a]_r$ its image in $E_r(A)$. For $[a]_r\in E_r(A)$,
we have $\delta_r([a]_r)=[da]_r$.

\subsection{Multicomplexes}

\begin{defi}
  A \emph{multicomplex} or \emph{$\infty$-multicomplex} $A$ is a bigraded $\kk$-module $A = \{A^{p,q}\}_{p,q\in\ZZ}$  endowed with a family of maps $\{d_i \colon A \to A\}_{i \ge 0}$ of bidegree $(-i,1-i)$ satisfying for all $l \ge 0$,
  \begin{equation}\label{E:multicomplex}
    \sum_{i+j=l} (-1)^i d_id_j=0.
  \end{equation}
  Let $n \ge 1$ be an integer.
  An \emph{$n$-multicomplex} is a multicomplex with $d_i=0$ for all $i\geq n$.

  For $1 \leq n \leq \infty$, a \emph{\textup(strict\textup) morphism} of $n$-multicomplexes is a map $f$ of bigraded $\kk$-modules of bidegree $(0,0)$ satisfying $d_if=fd_i$ for all $i\geq 0$.
  We denote by $\ncpx$ the category of $n$-multicomplexes and strict morphisms.
\end{defi}

\section{The category of spectral sequences}\label{sec:specseq}

\subsection{Definitions and basic properties}\label{subsec:defs}

\begin{defi}\label{def:spse_obj}
A \textit{spectral sequence} $(A,\psi)$ is a family of $r$-bigraded complexes $(A_r, d_r)$, for $r\geq 0$, together with a family of isomorphisms of bigraded $\kk$-modules $\psi_r: H_*(A_r)\rightarrow A_{r+1}$ for $r\geq 0$, called \textit{characteristic maps}.

A \textit{morphism of spectral sequences} is a family of morphisms $f_r:A_r\to B_r$ of $r$-bigraded complexes, for $r\geq 0$,
which is compatible with characteristic maps. We denote by $\spse$ the category of spectral sequences. 
\end{defi}

Note that $\spse$ is a subcategory of the product category $\prod_{r\geq 0} \rbc$ .

We will often omit the characteristic maps in the notation.
\smallskip

Note that a morphism of spectral sequences $f:(A,\psi)\rightarrow (B,\varphi)$ is completely determined by the $0$-page, $f_0:A_0\rightarrow B_0$, since $f_{i+1}=\varphi_i H_*(f_i)\psi_i^{-1}$, for all $i\geq 0$.
Furthermore, it is clear that the following proposition holds.

\begin{prop} The category of spectral sequences is an additive category. \qed
\end{prop}

\begin{rmk}
There are various conventions for spectral sequences. We have chosen ours to be compatible with the conventions for filtered complexes and multicomplexes in previous work on related model category structures in~\cite{CELW20, FGLW}. The differential on the $r$-page has bidegree $(-r, 1-r)$. Of course, it is straightforward to translate our results to the standard setting of a homological spectral sequence where the corresponding bidegree is $(-r, r-1)$ or that of a cohomological spectral sequence where it is $(r,1-r)$.
\end{rmk}

\begin{rmk}
The category of spectal sequences $\spse$ is neither complete nor cocomplete. Indeed, as in the following examples, 
cokernels and kernels do not exist in general in $\spse$. Thus it is not a pre-abelian category.
\end{rmk}

We write $R^{p,n}$ for the ring $R$ in bidegree $(p,n)$.
We denote by $R(p,n)$ the spectral sequence with the ring $R$ concentrated in bidegree $(p,n)$ and all differentials zero.

\begin{example}
\label{cocompletefalse}
Let $S$ be the spectral sequence given by
 $S_0=R^{0,0}\oplus R^{1,0}$ with $d_0=0$, $S_1=R^{0,0}\oplus R^{1,0}$ with $d_1=1_R: R^{1,0}\to R^{0,0}$
and $S_{\geq 2}=0$.
The morphism of spectral sequences $f:R(0,0)\rightarrow S$ determined by $f_0^{0,0}=f_1^{0,0}=1_R:R^{0,0}\to R^{0,0}$ has no cokernel.
\end{example}

\begin{example}
\label{completefalse}
Let $T$ be the spectral sequence given by $T_0=R^{0,0}\oplus R^{0,1}$ with $d_0=1_R: R^{0,0}\to R^{0,1}$
and $T_{\geq 1}=0$.
The morphism of spectral sequences $\pi:T\to R(0,0)$ such that $\pi_0^{0,0}=1_R:R^{0,0}\to R^{0,0}$ and $\pi_i=0$ for $i>0$ has no kernel.
\end{example}

\subsection{Pullback of surjections}

\begin{defi}
A morphism of spectral sequences $f:A\to B$ is called a \textit{surjection} if the morphism $f_r$ is bidegreewise surjective for every $r\geq 0$. 
We write $\Sur$ for the class of surjective morphisms in $\spse$.
\end{defi}

\begin{lem}
\label{lem:pbsurj}
The category of spectral sequences $\spse$ admits pullbacks of surjections along any map and this preserves surjections. 
Moreover, such pullbacks are computed pagewise.
\end{lem}

\begin{proof} Let
 \[
\xymatrix{
&A\ar[d]^p\\
U\ar[r]_g&B
}
\]
be a diagram of spectral sequences where $p$ is a surjection. For $m\geq 0$, let $X_m$ be the pullback in the category of $m$-bigraded complexes of the $m$-page of the spectral sequence. Since the category of $m$-bigraded complexes is abelian, and $p_m$ is surjective we have a short exact sequence of $m$-bigraded complexes

\[\xymatrix{0\ar[r] &X_m\ar[r]^{}& U_m\oplus A_m\ar[r]^-{g_m-p_m}& B_m\ar[r] &0}\]
which yields a long exact sequence in homology. The map $H_*(g_m-p_m)$ is isomorphic to the map $(g_{m+1}-p_{m+1}): U_{m+1}\oplus A_{m+1}\rightarrow B_{m+1}$, hence surjective, so that $H_*(X_m)$ is isomorphic to $\Ker(g_{m+1}-p_{m+1})=X_{m+1}$. Hence the collection $X=(X_m)_{m\geq 0}$ is a spectral sequence.

We claim that this is the pullback of the diagram in the category of spectral sequences. 
 For the universal property, given maps of spectral sequences $Y\to A$, $Y\to U$ making the diagram commute, we get a unique map of $m$-bigraded complexes $f_m: Y_m \to X_m$ making the diagram of $m$-pages commute, because
$X_m$ is the pullback on the $m$-page. Noting that the forgetful functor from $m$-bigraded complexes to bigraded modules preserves pullbacks, we see that $H(f_m)\cong f_{m+1}$  so that $f=(f_m)$ is the required unique map of spectral sequences. 

Note that 
\[
X_m=\{(u,a)\,|\, g_m(u)=p_m(a), u\in U_m, a\in A_m\},
\] 
so that the induced map $\pi_1: X\rightarrow U$ is a surjection. 
\end{proof}

\begin{rmk}\label{falsepullback} The category $\spse$ does not admit general pullbacks of epimorphisms, as shown by the following proposition and example.
\end{rmk}

\begin{prop}
A morphism $f:A\rightarrow B$ of spectral sequences such that $f_0:A_0\to B_0$ is surjective is an epimorphism.
\end{prop}

\begin{proof}
Let $i,j:B\rightarrow X$ be morphisms of spectral sequences such that $if=jf$. In particular, we have $i_0f_0=j_0f_0$ and since $f_0$ is surjective, we have $i_0=j_0$. Thus $i=j$.
\end{proof}

\begin{example}
Let us consider the morphism $\pi:T\rightarrow \kk(0,0)$ of Example \ref{completefalse}. It is an epimorphism because $\pi_0$ is surjective, but the pullback of $\pi$ along the map $0\rightarrow \kk(0,0)$ does not exist, because $\pi$ does not admit a kernel.
\end{example}

\section{Homotopy theory without model category structures}
\label{sec:abc}

The goal of this paper is to describe the homotopy theory of spectral sequences with respect to $E_r$-quasi-isomorphism (see Definition 
\ref{D:Erqiso}). We cannot expect to have a model category structure on the category of spectral sequences with this class of maps as the class of weak equivalences since we have 
seen that this category is neither complete nor cocomplete. 

Thus we will work with a weaker structure. Many variants are available in the literature; we will work with something close to what is known as a \emph{Brown category}.
In this section we introduce the homotopy theoretic material needed to achieve our goal.

\subsection{Almost Brown categories}

In this section we assume that the reader is familiar with the language of model categories, in particular with the notion of acyclic fibrations and fibrant objects.

\begin{defi}\label{D:ABCcat}
An \textit{almost Brown category}  is a category $\Cc$ with finite products and a final object $e$ together with two distinguished classes of 
maps called \textit{weak equivalences}
$(\Ww)$ and \textit{fibrations} ($\Fib$), satisfying the following axioms.
\begin{itemize}
\item[($A$)] Let $f$ and $g$ be composable morphisms. If any two of $f$, $g$ and $gf$ are weak equivalences, then so is the third. (That is, the class of weak equivalences satisfies the two-out-of-three property.)  All 
isomorphisms are weak equivalences.
\item[($B$)] The composite of two fibrations is a fibration. All isomorphisms are fibrations.
\item[($C$)] The pullback of an acyclic fibration along any map exists and is an acyclic fibration.
\item[($D$)]  Any morphism $u:X\rightarrow Y$ in $\Cc$ can be factored $u=pi$ with $p$ a fibration and $i$ right inverse to an acyclic 
fibration. 
\item[($E$)] Any object of $\Cc$ is fibrant.
\end{itemize}
In addition, if axiom $(D)$ holds functorially, we will say that $\Cc$ is an \emph{almost Brown category with functorial factorization}.
\end{defi}

\begin{defi}\label{D:leftexact} A functor $F:\Cc\rightarrow\Dd$ between almost Brown categories is called \textit{left exact} if it preserves finite products, the class of fibrations, the class of acyclic fibrations and pullback of acyclic fibrations.
\end{defi}

We recall that given an object $B$ of $\Cc$, a \emph{path space} for $B$ is an object $B^I$ together with maps
\[\xymatrix{B\ar[r]^\iota &B^I\ar[r]^-{(\partial_{-},\partial_{+})}& B\times B}\]
where $\iota$ is  a weak equivalence, $(\partial_{-},\partial_{+})$ a fibration and the composite is the diagonal map. 
Note that axiom $(D)$ of Definition \ref{D:ABCcat} implies the existence of a path space for any object in an almost Brown category.
\medskip

We next recall the original definition of K.S.~Brown in \cite{Brown}.

 \begin{defi}
 \label{def:bc}
A \textit{Brown category} is a category with finite products and a final object $e$ together with two distinguished classes of maps called \emph{weak equivalences} 
$(\Ww)$ and \emph{fibrations} ($\Fib$), satisfying the following axioms.
\begin{itemize}
\item[($A$)] Let $f$ and $g$ be composable morphisms. If any two of $f$, $g$ and $gf$ are weak equivalences, then so is the third.  All isomorphisms are weak equivalences.
\item[($B$)]  The composite of two fibrations is a fibration. All isomorphisms are fibrations.
\item[($C'$)] The pullback of a fibration along any map exists and is a fibration. The pullback of an acyclic fibration along any map exists and is an acyclic fibration.
\item[($D'$)] For any object $B$ there exists at least one path space $B^I$.
\item[($E$)] Every object is fibrant.
\end{itemize}
\end{defi}

\begin{rmk}\label{R:factorizationlemma}  In a Brown category $\Cc$, axiom $(D')$ is equivalent to axiom $(D)$. This is due to the factorization lemma, 
which is proved by using the axiom that the pullback of a fibration along any map exists and is a fibration. Concretely, any morphism $u:A\rightarrow B$ factorizes as
\[\xymatrix{A\ar[r]^-{(1_A,\iota u)}\ar@/_1pc/[rr]_{u}& A\times_B B^I\ar[r]^-{\partial_+\pi_2} & B}\]
where the object $A\times_B B^I$ is called the mapping path space of $u$, the first map is a weak equivalence right inverse to an acyclic fibration and the second map is a fibration.
\end{rmk}

The following corollary is a direct consequence of the remark above.

\begin{cor} A Brown category is an almost Brown category. \qed
\end{cor}

\begin{rmk} If $\Cc$ is a model category, then the subcategory $\Cc^f$ of fibrant objects of $\Cc$ is a Brown category. If all objects of $\Cc$ are fibrant, then it is a Brown category, hence an almost Brown category. A right Quillen functor between two model categories whose objects are all fibrant is a left exact functor in the sense of Definition \ref{D:leftexact}.
\end{rmk}

Unfortunately, in our examples we do not have all the axioms of a Brown category, since we usually do not have pullbacks of fibrations, however we have path objects and the factorization induced by them, that is the mapping path space of a morphism.

\subsection{Comparison with partial Brown categories of fibrant objects}
\label{sec:BPC}
In~\cite{Horel}, Horel introduced the notion of a partial Brown category of cofibrant objects and  in Remark 2.4 of loc.~cit.~it is noted that all the results dualize to the case of interest for us. This gives a setting for homotopy theory closely related to the one we have presented above and we compare the two here.

We start by making explicit the dual to Horel's Definition 2.2. The notation $\Cc^{[1]}$ denotes the arrow category of $\Cc$.

\begin{defi}
\label{def:PBC}
A \emph{partial Brown category of fibrant objects} is a category $\mathcal{C}$, with two subcategories $w\Cc$ and $f\Cc$ whose maps are called respectively the \textit{weak equivalences} and \textit{acyclic fibrations} such that the following axioms are satisfied.
\begin{enumerate}
\setlength{\itemindent}{-1em}
\item Both $w\Cc$ and $f\Cc$ contain the isomorphisms of $\Cc$ and $f\Cc$ is contained in $w\Cc$.
\item The weak equivalences satisfy the two-out-of-three property.
\item The pullback of an acyclic fibration along any map exists and is an acyclic fibration.
\item There are three functors $f, w, s$ from $w\Cc^{[1]}\to w\Cc^{[1]}$ such that for each weak equivalence $g$ we have
$g=f(g)\circ w(g), s(g)\circ w(g)=1$ and $f(g)$ and $s(g)$ are in $f\Cc^{[1]}$.
\end{enumerate}
\end{defi}

The following proposition is immediate.
\begin{prop} If $\Cc$ is an almost Brown category with functorial factorization, then $\Cc$ is a partial Brown category of fibrant objects. \qed
\end{prop}

\begin{defi}
\label{defi:leftexactPBC}
Let $\Cc$ and $\Dd$ be partial Brown categories of fibrant objects. A functor $\Cc\to \Dd$ is called \emph{left exact} if it 
preserves weak equivalences, acyclic fibrations and pullbacks of acyclic fibrations.
\end{defi}

\begin{prop}
If  $\Cc$ and $\Dd$ are almost Brown categories with functorial factorization, then a left exact functor $F:\Cc\to \Dd$ is also
 left exact as a functor
of partial Brown categories.
\end{prop}

\begin{proof}
We need to check that $F$ preserves weak equivalences.
Let $u$ be a weak equivalence and factorize this as $u=pi$ with $p$ an acyclic fibration and $i$ right inverse to an acyclic fibration. 
Then $F(u)=F(p)F(i)$ and since $F$ preserves acyclic fibrations, $F(p)$ is an acyclic fibration and $F(i)$ is  right inverse to an acyclic fibration. This 
implies that $F(i)$ is a weak equivalence and thus so is $F(u)$. 
\end{proof}

\section{Almost Brown category structures on spectral sequences}
\label{sec:main results}

In this section we again fix an integer $r\geq 0$.

\subsection{$E_r$-quasi-isomorphisms and $r$-fibrations}

\begin{defi}\label{D:Erqiso}
A morphism of spectral sequences $f:A\to B$ is called an \textit{$E_r$-quasi-isomorphism} if
the morphism $f_r:A_r\to B_r$ is a quasi-isomorphism of $r$-bigraded complexes, or equivalently if the morphisms $f_k$ are isomorphisms for $k>r$.

A morphism of spectral sequences $f:A\to B$ is called an \emph{$r$-fibration} if the morphisms $f_k$ are surjective for $0\leq k\leq r$.
\end{defi}

We denote by  $\Ee_r$ the class of $E_r$-quasi-isomorphisms of $\spse$. This class
contains all isomorphisms of $\spse$ and
satisfies the two-out-of-three property. 

We denote by $\Fib_r$ the class of $r$-fibrations of $\spse$. This class contains all isomorphisms and is stable under composition.

Note that acyclic fibrations are those maps that are surjective at the $k$-page of the spectral sequence for $k\leq r$ and isomorphisms for $k>r$. In particular the class of acyclic fibrations coincides with the class of surjective $E_r$-quasi-isomorphisms, that is
$\Ee_{r}\cap \Fib_r =\Ee_r\cap \cap_s \Fib_s=\Ee_r\cap \Sur$.

It is clear that we have inclusions:
\[ \Ee_{r}\subset \Ee_{r+1}, \qquad \Fib_{r+1}\subset \Fib_r\qquad \text{and} 
\qquad \Ee_{r}\cap \Fib_r \subset \Ee_{r+1}\cap \Fib_{r+1}
\]
for all $r\geq 0$.

\subsection{Mapping path space construction for spectral sequences}\label{S:mapping}

In this section we define a functorial $r$-path and an explicit $r$-mapping path space for any morphism in  the category of spectral sequences.

\begin{defi}\label{def:lambda}
Let $\Lambda_r$ be the spectral sequence $Re_-\oplus Re_+\oplus Ru$ where $e_{\pm}$ are in bidegree $(0,0)$ and $u$ is in bidegree
$(-r,1-r)$, with all differentials zero except at the $r$-page of the spectral sequence where $d_r(e_-)=-u$, $d_r(e_+)=u$. The $r+1$-page of the spectral sequence is then concentrated in bidegree $(0,0)$ with a single $R$-module, free of rank 1, generated by $e_++e_-$.
\end{defi}

Note that we can consider $\Lambda_r\otimes A$ for any spectral sequence $A$ and that this is again a spectral sequence.
\medskip

We next define a collection of functorial paths indexed by an integer $r\geq 0$ on the category of spectral sequences,
giving rise to the corresponding notions of $r$-homotopy.

\begin{defi}
\label{def:rpathob}
The \textit{$r$-path of a spectral sequence $A$} is the spectral sequence $P(r;A)= \Lambda_r\otimes A$. Explicitly, the pages of the spectral sequence $P(r;A)$ are given by

\[P_m(r;A)^{p,q}:=\begin{cases} A_m^{p,q}\oplus A_m^{p+r,q+r-1}\oplus A_m^{p,q},&\text{if } 0\leq m\leq r\\
A_m^{p,q},&\text {if } m>r\end{cases}\]
with the differentials $D_m:P_m(r;A)\to P_m(r;A)$ of bidegree $(-m,1-m)$ given by
\[D_m:=\left(
\begin{matrix}
 d_m&0&0\\
 0&(-1)^{m+r+1}d_m&0\\
 0&0&d_m
\end{matrix}
\right)\text{ for }m< r,\ 
D_r:=\left(
\begin{matrix}
 d_r&0&0\\
 -1&-d_r&1\\
 0&0&d_r
\end{matrix}
\right)\]
and $D_m=d_m$ for $m>r$.
\end{defi}

We have a factorisation of the diagonal map
\[\xymatrix{\kk\ar[r]^-{\iota}&\Lambda_r \ar[r]^-{(\partial_-, \partial_+)}&\kk \times \kk}\]
and thus morphisms of spectral sequences
\[\xymatrix{A\ar[r]^-{\iota_A}&P(r;A) \ar@<1ex>[r]^{\partial^+_A} \ar@<-1ex>[r]_{\partial^-_A}&A}\,\,\,;\,\,\, \partial^\pm_A\circ \iota_A=1_A,\]
given by $\partial^-_A(x,y,z)=x$, $\partial^+_A(x,y,z)=z$ and $\iota_A(x)=(x,0,x)$ on the $m$-page of the spectral sequence for $m\leq r$ and by the identity maps for $m>r$. 
We will often omit the subscripts of these maps when there is no danger of confusion.

The use of the term $r$-path is justified below. In particular, 
  $\iota_A$ is an $E_r$-quasi-isomorphism. Furthermore $(\partial_A^+,\partial_A^-): P(r;A)\rightarrow A\times A$ is an $r$-fibration. In addition $\partial_A^+$ and $\partial_A^-$ are acyclic $r$-fibrations.

\begin{defi}
\label{def:rpathmor}
The \textit{$r$-path of a morphism} $f:(A,d_m^A)\to (B,d_m^B)$ of spectral sequences
is the morphism of spectral sequences $P(r;f):(P(r;A),D_m^A)\to (P(r;B),D_m^B)$ given by
\[P(r;f)_m:=(f_m, (-1)^{m}f_m,f_m),\]
for $m\leq r$ and $P(r;f)=f_m$ for $m>r$.
\end{defi}
The above definitions give rise to a functorial path $P(r;-):\spse\to \spse$ in the category of spectral sequences.

We would like to use this for the factorization of any morphism $u:A\rightarrow B$ of spectral sequences in the spirit of Remark \ref{R:factorizationlemma}.
We remark that the $r$-mapping path space   $\overline{P}(r;u):=A\times_B P(r;B)$ of $u$  exists and takes  the following form
 \[\overline{P}(r;u)_m^{p,q}:= \begin{cases}  A_m^{p,q}\oplus B_m^{p+r,q+r-1}\oplus B_m^{p,q},&\text{if } 0\leq m\leq r,\\
A_m^{p,q},&\text {if } m>r,\end{cases}\]
with differentials $D_m:\overline{P}(r;u)_m\to \overline{P}(r;u)_m$ of bidegree $(-m,-m+1)$ given by
\[D_m=\left(
\begin{matrix}
 d_m^A&0&0\\
 0&(-1)^{m+r+1}d_m^B&0\\
 0&0&d_m^B
\end{matrix}
\right)\text{ for }m< r,\quad
D_r=\left(
\begin{matrix}
 d^A_r&0&0\\
 -u_r&-d^B_r&1_B\\
 0&0&d^B_r
\end{matrix}
\right)
\]
and $D_m$ is induced by $d^A_m$ for $m>r$.
The factorization of $u$ as
\[\xymatrix{A\ar[r]^-i & \overline{P}(r;u)\ar[r]^-p & B}\]
takes the following form
\begin{align*}
i_m(a)&=\begin{cases} (a,0,u_m(a)), &\text{if } 0\leq m\leq r,\\
a, &\text {if } m>r\end{cases} \quad\text{and}\\
p_m(a,b',b)&= \begin{cases} b, &\text{if } 0\leq m\leq r,\\
u_m(a), &\text {if } m>r.\end{cases}
\end{align*}
The map $i$ is right inverse to an acyclic fibration, namely  the projection of $\overline{P}(r;u)$ onto $A$.
It is clear from the formulas that $p$ is an $r$-fibration and that the factorization is functorial.

\subsection{Homotopy theory of spectral sequences}

\begin{teo}\label{T:fund} The category of spectral sequences together with the class $\Ee_r$ of $E_r$-quasi-isomorphisms and the class $\Fib_r$ of $r$-fibrations is an almost Brown category with functorial factorization. Hence it is a partial Brown category of fibrant objects.
\end{teo}

\begin{proof} Axioms $(A), (B)$ and $(E)$ are clearly satisfied. Let us show axiom $(C)$. Let $p:C\rightarrow B$ be an acyclic $r$-fibration and let $u: A\rightarrow B$ be  a morphism in $\spse$. Any acyclic $r$-fibration is surjective, and by Lemma \ref{lem:pbsurj}, the pullback of $p$ along $u$ exists and is the spectral sequence $X$ whose $m$-page is described by $X_m=\{(a,c)\in A_m\times C_m)\, |\, u_m(a)=p_m(b)\}$ with induced map $\pi_1:X\rightarrow A$  the projection to the first factor. The proof of that lemma shows that if $p_r$ is a quasi-isomorphism so is the $r$-page  $(\pi_1)_r$ of $\pi_1$.  Axiom $(D)$ follows from the mapping
path space construction of Section \ref{S:mapping}. \end{proof}

\begin{notation}
\label{nota:spse_r}
We write $(\spse)_r$ for the  almost Brown category of spectral sequences with the structure specified in Theorem~\ref{T:fund}.
\end{notation}

Recall that we have inclusions:
\[ \Ee_{r}\subset \Ee_{r+1}, \qquad \Fib_{r+1}\subset \Fib_r\qquad \text{and} 
\qquad \Ee_{r}\cap \Fib_r \subset \Ee_{r+1}\cap \Fib_{r+1}
\]
for all $r\geq 0$.
Thus, for $r\leq s$, the identity functor $\Id:(\spse)_r \to (\spse)_s$ preserves weak equivalences and acyclic fibrations, but not fibrations.
Therefore it is not left exact as a functor of almost Brown categories, but it is left exact when viewed as a functor of the corresponding partial Brown 
categories.

\subsection{$r$-homotopies}\label{S:r-homotopy}

As in classical homotopy theory, the functorial $r$-path $P(r;-):\spse\to \spse$ yields  a natural notion of homotopy: for $f,g:A\to B$  two morphisms of spectral sequences, an \textit{$r$-homotopy from $f$ to $g$} is given by a morphism 
of spectral sequences $h:A\to P(r;B)$ such that $\partial^-_B\circ h=f$ and $\partial^+_B\circ h=g$.
We use the notation $h:f\simr{r} g$.  An \textit{$r$-homotopy equivalence} is a morphism of spectral sequences $f:A\to B$ such that there exists a morphism $g:B\to A$ satisfying $f\circ g\simr{r} 1_B$ and $g\circ f\simr{r} 1_A$.

\begin{prop}\label{equivrelTC}
The notion of $r$-homotopy defines an equivalence relation on the set of morphisms between two given spectral sequences,
which is compatible with the composition.
\end{prop}

\begin{proof} Unravelling the definition we have that if $f,g:(A,d_m^A)\to (B,d_m^B)$ are two morphisms of spectral sequences, then $f\simr{r} g$ if and only if there exists
a collection of morphisms $\widehat h_m:A_m\to B_m$ of bigraded modules, of
bidegree $(r,r-1)$, for all $0\leq m\leq r$, satisfying
\[
\begin{cases}
(-1)^{m+r+1}d_m^B \widehat h_m-\widehat h_m d_m^A =0,& \text{ if } 0\leq m<r, \\
-d_r^B   \widehat h_r-\widehat h_r d_r^A =f_r-g_r,&
\end{cases}
\]
and $H_*(\widehat h_m)=\widehat h_{m+1}$ for $0\leq m<r$.
The proposition then follows.
\end{proof}

Denote by $\Ss_r$ the class of $r$-homotopy equivalences of $\spse$. 
This class is closed under composition and contains all isomorphisms. 
In addition, we have $\Ss_r\subseteq \Ss_{r+1}$ and $\Ss_r\subseteq \Ee_r$, for all $r\geq 0$.

\begin{prop}\label{QuotientcatTC}
The localized category $\spse[\Ss_r^{-1}]$
is canonically isomorphic to the quotient category  
$\spse/\!\!\simr{r}$.
\end{prop}

\begin{proof} The proof is classical, and requires that the morphism from a spectral sequence $A$ to its path space $\iota_A:(A,d_m)\lra (P(r;A),D_m)$, is not only an $E_r$-quasi-isomorphism but also an $r$-homotopy equivalence. We prove this statement.
Recall that $(\iota_A)_m(x)=(x,0,x)$ for $m\leq r$ and $(\iota_A)_m=1_A$ for $m>r$. 
Since $\partial_A^-\iota_A=1_{A}$, it suffices to define an $r$-homotopy from $1_{P(r;A)}$ to
$\iota_A\partial_A^-$.
Consider the morphism $\widehat h_m:P(r;A)\to P(r;A)$ of bidegree $(r,r-1)$ defined by
$\widehat h_m(x,y,z)=(0,0,-y)$ for $0\leq m\leq r$.
It is clear that for $m<r$ we have $(-1)^{m+r-1}D_m\widehat h_m-\widehat h_mD_m=0$, $H_*(\widehat h_m)=\widehat h_{m+1}$ and
\[
(-D_r\widehat h_r-\widehat h_rD_r)(x,y,z)=(0,y,d_ry)+(0,0,-d_ry-x+z)=(x,y,z)-(x,0,x).
\]
As in the proof of Proposition \ref{equivrelTC} this implies that $\iota_A$ is an $r$-homotopy equivalence.
\end{proof}

In section~\ref{sec:comparisons}, we will compare this notion of homotopy for spectral sequences with notions for filtered complexes and multicomplexes.

\subsection{Generation of $r$-fibrations and acyclic $r$-fibrations}

This section is devoted to the description of $r$-fibrations and acyclic $r$-fibrations as maps having the right lifting property with respect to a set of morphisms in $\spse$.
We adopt the language of model categories. For $I$ a class of maps in $\spse$, we say that a morphism of spectral sequences $f$ is $I$-injective if it has the right lifting property with respect to $I$.

To describe the generating sets, we first introduce some basic objects.

\begin{defi}\label{def:discs} 
Let $p,n\in\ZZ$. For all $r\geq 0$, let $\Dd_r(p,n)$ be the spectral sequence defined as follows:
\[
\begin{cases}
\Dd_r(p,n)_i=\kk^{p,n}\oplus \kk^{p-r,n+1-r}, &d_i=0, \text{ for } 0\leq i<r, \\
\Dd_r(p,n)_r=\kk^{p,n}\stackrel{1}{\lra}\kk^{p-r,n+1-r},\\
\Dd_r(p,n)_i=0, &\text{for } i>r.
\end{cases}
\]

For all $r\geq 1$ define
\[\Ss_r(p,n):=\Dd_{r-1}(p-1,n-1)\oplus \Dd_{r-1}(p+r-1,n+r-2)\]

For all $r\geq 1$ define a morphism of spectral sequences
\[\varphi_r: \Dd_{r}(p,n)\lra \Ss_{r}(p,n)\]
via the identity on $\kk$ whenever it is bigradedly defined.
\end{defi}

\begin{defi} Let $(A,\varphi)$ be a spectral sequence. 
\begin{enumerate}
\item A sequence of elements $(a_0^{p,n},\ldots,a_{m+1}^{p,n})$ with $a_i^{p,n}\in A_i^{p,n}$ is said to be 
\emph{compatible} if for every $0\leq i\leq m$, $d_ia_i^{p,n}=0$ and $a_{i+1}^{p,n}=\varphi_i([a_i^{p,n}])$ where $[a_i^{p,n}]$ is the class of $a_i^{p,n}$ in $H_*(A_i)$.
\item Denote by $D_r^{p,n}(A)$ the $\kk$-submodule of $A^{\times (2r+2)}$ consisting of pairs
\[(a_0^{p,n},\ldots,a_r^{p,n});(b_0^{p-r,n+1-r},\ldots,b_r^{p-r,n+1-r})\] of compatible sequences satisfying $d_ra_r^{p,n}=b_r^{p-r,n+1-r}$. This yields a functor, denoted by $D_r^{p,n} : \spse\rightarrow \modu$.
\end{enumerate}
\end{defi}
The following proposition is a direct consequence of the definitions.

\begin{prop}\label{P:morphismfromdisks} Let $(A,\varphi)$ be a spectral sequence.
\begin{enumerate}
\item There is a one-to-one correspondence between infinite compatible sequences $(a_0^{p,n},\ldots,a_m^{p,n},\ldots)$ and morphisms of spectral sequences $R(p,n)\rightarrow A$.
\item We have $D_r^{p,n}=\Hom_{\spse}(\Dd_r(p,n),-)$, that is, $D_r^{p,n}$ is represented by $\Dd_r(p,n)$.\qed
\end{enumerate} 
\end{prop}

\begin{defi}
Let $I_r$ and $J_r$ be the sets of morphisms of $\spse$ given by 
\[I_r:=\left\{\varphi_{r+1}:\Dd_{r+1}(p,n)\lra \Ss_{r+1}(p,n)\right\}_{p,n\in\ZZ}\text{ and }
J_r:=\left\{0\lra \Dd_{r}(p,n)\right\}_{p,n\in\ZZ}.\]

Let $I'_r$ and $J_r'$ be the sets of morphisms of $\spse$ given by 
\[I_r':=\cup_{k=0}^{r-1} J_k\cup I_r\text{ and }
J_r':=\cup_{k=0}^r J_k.\]

\end{defi}

\begin{prop}\label{J_r}A morphism of spectral sequences is an $r$-fibration if and only if  it has the right lifting property with respect to $J'_r$.
\end{prop}
\begin{proof} Let $f:(A,\varphi)\rightarrow (B,\psi)$ be a morphism of spectral sequences.
It is clear that if $f$ is $J_r$-injective then $f_r$ is bidegreewise surjective. It is also clear that $f$ is $J_0$-injective if and only if $f_0$ is bidegreewise surjective.
Assume that for every $0\leq i\leq r$ we have that $f_i$ is bidegreewise surjective.
Let $(b_0,\ldots,b_r;b'_0,\ldots, b'_r)$ in $D_r^{p,n}(B)$. 
Since $f_r$ is surjective, there exists $a_r\in A_r$ such that $f_r(a_r)=b_r$ hence $f_r(d_ra_r)=b'_r$. We set $a'_r=d_ra_r$. 
Pick $u_{r-1}$ a cycle in $A_{r-1}$ such that $\varphi([u_{r-1}])=a_r$. Hence $\psi([f_{r-1}(u_{r-1})])=f_r(a_r)=b_r=\psi[b_{r-1}]$  and there exists $y\in B_{r-1}$ such that $f_{r-1}(u_{r-1})=b_{r-1}+d_{r-1}y$.  And $y=f_{r-1}(x)$  for some $x$ since $f_{r-1}$ is surjective. Hence $b_{r-1}=f_{r-1}(u_{r-1}-d_{r-1}x)$ and $a_{r-1}=u_{r-1}+d_{r-1}(x)$ satisfies the required conditions.
By induction, we obtain that there exists $(a_0,\ldots,a_r;a'_0,\ldots,a'_r)$ in $D_r^{p,n}(A)$ such that for all $i$ we have $f_i(a_i)=b_i$ and $f_i(a'_i)=b'_i$, giving the required lift.
\end{proof}

Note that this proposition can be stated as $f:A\rightarrow B$ is an $r$-fibration if and only if for every $0\leq k\leq r$ and for every $p,n\in\ZZ$, $D_k^{p,n}(f)$ is surjective.

\begin{prop}\label{I_r} A morphism of spectral sequences is an acyclic $r$-fibration if and only if  it has the right lifting property with respect to $I'_r$.
\end{prop}
\begin{proof}
 Assume first that $f:A\rightarrow B$ is $I'_r$-injective. Then $f_k$ is bidegreewise surjective for every $0\leq k\leq r-1$. Let us show that $f_r$ is surjective on cycles. 
 Let $b_{r}\in B_{r}^{p+r,n+r-1}$ be such that $d_{r}(b_r)=0$. 
 One can then build a sequence $(b_0,\ldots, b_{r-1}, b_r; 0,\ldots, 0)$ in $D_r^{p+r,n+r-1}(B)$ which yields
 a commutative diagram
  \[
 \xymatrix{
 \Dd_{r+1}(p,n)\ar[d]_{\varphi_{r+1}}\ar[r]^-{0}&A\ar[d]^{f}\\
 \Ss_{r+1}(p,n)\ar[r]_-{b_r}&B } \]
This diagram admits a lift, giving rise to an element $(a_0,\ldots,a_r;0,\ldots,0)$ in $D_r^{p+r,n+r-1}(A)$ satisfying $d_ra_r=0$ and $f_i(a_i)=b_i$ for $0\leq i\leq r$. This proves that $f_r$ is surjective on cycles, and thus that $f_{r+1}$ is surjective.

Let us show that $f_r$ is surjective. Let $b_{r}\in B_{r}^{p+r,n+r-1}$. One can choose compatible sequences: $(b_0,\ldots,b_r;b'_0,\ldots,b'_r)$ in
$D_r^{p+r,n+r-1}(B)$ and from the first part a lift $a'=(a'_0,\ldots,a'_r)$ of $(b'_0,\ldots,b'_r)$ since $d_rb'_r=d_rd_rb_r=0$,
which yields again a commutative diagram
\[\xymatrix{
 \Dd_{r+1}(p,n)\ar[d]_{\varphi_{r+1}}\ar[r]^-{a'}&A\ar[d]^{f}\\
 \Ss_{r+1}(p,n)\ar[r]_-{b_r}&B } \]
admitting a lift. This gives an element $a_r\in A_r$ such that $d_ra_r=a'_r$ and $f_r(a_r)=b_r$.
As a consequence $f$ is $J'_r$-injective, and thus an $r$-fibration. 

We have proved that $f_{r+1}$ is surjective.
Let us show that $f_{r+1}$ is injective. Let $a_{r+1}\in A_{r+1}$ be such that $f_{r+1}(a_{r+1})=0$ and $(a_0,\ldots,a_{r+1};a'_0,\ldots,a'_{r+1})\in D_{r+1}(A)$ which represents $a_{r+1}$. Since $f_{r+1}(a_{r+1})=0$, we have $f_{r+1}(a'_{r+1})=0$ and there exist $b_r, c_r\in B_r$ such that $d_r(b_r)=f(a_r)$ and $d_r(c_r)=f_r(a'_r)$. This induces the following diagram denoted $(\ast)$
\[\xymatrix{
 \Dd_{r+1}(p,n)\ar[d]_{\varphi_{r+1}}\ar[r]^-{a_{r+1}}&A\ar[d]^f\\
 \Ss_{r+1}(p,n)\ar[r]_-{(c_r,b_r)}&B } \]
which admits a lift. In particular, there exists $\alpha_r$ with $d_r\alpha_r=a_r$. Hence $[a_r]=[0]$, that is $a_{r+1}=0$.
 
Conversely, assume $f$ is an acyclic $r$-fibration.  Consider the diagram $(\ast)$. Since $f_{r+1}(a_{r+1})=0$ and $f_{r+1}$ is an isomorphism we deduce that $a_{r+1}=0$ so that $a_r$ is a boundary, as well as $a'_{r}$. We conclude that a lift exists using the fact that $D_r(f)$ is surjective.
\end{proof}

\section{Comparisons with filtered complexes and multicomplexes}
\label{sec:comparisons}

In this section we compare $(\spse)_r$ with corresponding structures on filtered complexes and  multicomplexes. 
 In previous work~\cite{CELW20} we have established model category structures on these categories where the weak equivalences
are the $E_r$-quasi-isomorphisms.
 Thus we are able to compare the underlying almost Brown category structures with $(\spse)_r$.

\subsection{Filtered complexes}

Let $\fcpx$ be the category of filtered complexes  and let $(\fcpx)_r$ denote this category with the $r$-model structure of~\cite[Theorem 3.16]{CELW20}.
We consider it as an almost Brown category where the weak equivalences are the $E_r$-quasi-isomorphisms and the fibrations are the maps such that $Z_0(f)$ is surjective and $E_i(f)$ is surjective for $0\leq i\leq r$.

\begin{prop}
\label{prop:compare-fcx}
The spectral sequence functor $E:(\fcpx)_r\to (\spse)_r$ preserves weak equivalences and is a left exact functor of almost Brown categories.
\end{prop}

\begin{proof}
It is clear that $E$ preserves finite products, weak equivalences, fibrations and acyclic fibrations. For the pullback condition, consider the diagram:
 \[
\xymatrix{
&A\ar[d]^p\\
U\ar[r]_g&B
}
\]
in $(\fcpx)_r$ where $p$ is an acyclic fibration. The pullback in $\fcpx$ is 
\[X=\Ker(p-g:U\oplus A\to B)=\{(u,a)\, |\, g(u)=p(a), u\in U, a\in A\},
\]
with
$d(u,a)=(du,da)$, since $\fcpx$ is a pre-abelian category.
Using the surjectivity of $E_i(p)$ for all $i$, the same proof as in Lemma~\ref{lem:pbsurj} shows that the associated spectral sequence has
\[
E(X)_i=\{(u,a)\,|\, E(g)(u)=E(p)(a), u\in E(U)_i, a\in E(A)_i\}.
\]
That is, it has the pagewise pullback of $i$-bigraded complexes as its $i$-page and, by Lemma~\ref{lem:pbsurj}, this is the pullback in $\spse$.
\end{proof}

Recall the notion of $r$-homotopy between morphisms of filtered complexes from Definition~\ref{def:rhtpy_fcpx}.

\begin{prop}
The spectral sequence functor $E:(\fcpx)_r\to (\spse)_r$ preserves $r$-homotopy.
\end{prop}

\begin{proof}
The notion of $r$-homotopy between morphisms $f,g:A\to B$ of filtered complexes can
be formulated in terms of a version $\Lambda_r^{FC}$ of  $\Lambda_r$ in filtered complexes. Let $\Lambda_r^{FC}=
Re_-\oplus Re_+\oplus Ru$ where $e_-, e_+$ are in degree $0$ and filtration $0$ and $u$ is in 
degree $1$ and filtration $-r$. The differential is determined by $d(e_-)=-u$, $d(e_+)=u$. And we have morphisms 
$\partial^-, \partial^+:\Lambda_r^{FC}\to R$ given by projection to $Re_-$ and $Re_+$ respectively.
Then giving an $r$-homotopy from $f$ to $g$ is equivalent to giving a morphism 
of filtered complexes $h:A\to \Lambda_r^{FC}\otimes B$ such that $\partial^-_B\circ h=f$ and $\partial^+_B\circ h=g$.

The associated spectral sequence $E(\Lambda_r^{FC})$ is $\Lambda_r$ as in Definition~\ref{def:lambda} and more generally
$E(\Lambda_r^{FC}\otimes A)\cong \Lambda_r\otimes E(A)$. Thus an $r$-homotopy $h$ between $f$ and $g$ gives rise to
an $r$-homotopy $E(h)$ between $E(f)$ and $E(g)$.
\end{proof}

\subsection{Multicomplexes}

Recall that we write $\ncpx$ for the category of $n$-multicomplexes and strict morphisms.
Here $2\leq n\leq \infty$, where the case $n=\infty$ is the category of multicomplexes.
An $n$-multicomplex has an associated functorial spectral sequence, described explicitly in~\cite{LWZ}.
Indeed there is a totalization functor to filtered complexes and then we take the associated spectral sequence.
That is, we have a commutative diagram:
\[
\begin{tikzcd}
\ncpx \arrow[r, "E'"]\arrow[d, "\Tot"] &\spse \\
\fcpx\arrow[ur, "E"'] &
\end{tikzcd}
\]

Note that we write $E'=E\circ\Tot$ for the composite functor, 
but we will often drop the dash and just write $E_i$ for the pages of the spectral sequence associated to a multicomplex.

We write $(\ncpx)_r$ for the category of $n$-multicomplexes and strict morphisms with the $r$-model structure of~\cite[Theorem 3.30]{FGLW}.
We use the same notation for the corresponding almost Brown category where the weak equivalences are the $E_r$-quasi-isomorphisms 
and the fibrations are the maps $f$ such that $E_i(f)$ is surjective for $0\leq i\leq r$.
\medskip

\begin{prop}
\label{prop:compare-mcx}
The spectral sequence functor $E':(\ncpx)_r\to (\spse)_r$ preserves weak equivalences and
 is a left exact functor of almost Brown categories.
\end{prop}

\begin{proof}
It is clear that $E'$ preserves finite products, weak equivalences and fibrations. For the pullback condition, consider the diagram:
 \[
\xymatrix{
&A\ar[d]^p\\
U\ar[r]_g&B
}
\]
in $\ncpx$. The pullback in $\ncpx$ exists and it is 
\[
X=\Ker(p-g:U\oplus A\to B)=
\{(u,a)\,|\, g(u)=p(a), u\in U, a\in A\},
\] with
$d_i(u,a)=(d_iu,d_ia)$ for all $i\geq 0$. Indeed, the category $\ncpx$ has a description as a module category given in~\cite[Proposition 4.4]{FGLW}
and so it is abelian.

Let $Y_i$ denote the pullback in $i$-bigraded complexes of
\[
\xymatrix{
&E_i(A)\ar[d]^{E_i(p)}\\
E_i(U)\ar[r]_{E_i(g)}&E_i(B)
}
\]
and note that $E_0(X)\cong Y_0$ as $0$-bigraded complexes.

Now suppose that $p$ is an acyclic fibration, in particular $E_i(p)$ is surjective for all $i$, and assume that $E_n(X)\cong Y_n$ as $n$-bigraded complexes. 

As in Lemma~\ref{lem:pbsurj}, we have $Y_{n+1}\cong\ker(E_{n+1}(p)-E_{n+1}(g))$ and the argument of that proof also shows that 
we have an isomorphism of underlying bigraded $R$-modules  $E_{n+1}(X)\cong Y_{n+1}$. It remains to check that this can be upgraded
to an isomorphism of $(n+1)$-bigraded complexes and this can be seen from the explicit description of the differentials in the spectral sequence of a multicomplex
in~\cite{LWZ}.

Then $E'(X)$ has the pagewise pullback of $i$-bigraded complexes as its $i$-page and, by  Lemma~\ref{lem:pbsurj}, this is the pullback in $\spse$.
\end{proof}

\begin{rmk}
Note that  the proof shows that $E'$ preserves pullbacks along any map of any map $p$ such that $E_i(p)$ is surjective for all $i$.
\end{rmk}

\begin{rmk}
Note that in this multicomplex case $E'$ also reflects the weak equivalences and fibrations.
\end{rmk}

\begin{prop}
\label{prop:Erpath}
For $n=\infty$, the spectral sequence functor $E':(\ncpx)_r\to (\spse)_r$ preserves the $r$-path.
\end{prop}

\begin{proof}
The $r$-path for multicomplexes was defined in~\cite[Definition 3.14]{CELW18}. From the explicit description of the spectral sequence of a multicomplex, it is straightforward to see that the spectral sequence corresponding to the multicomplex $P_r(A)$ is $P(r;E(A))$.
We have $E(\iota_A)=\iota_{E(A)}$, $E(\delta^-_B)=\delta^-_{E(B)}$ and $E(\delta^+_B)=\delta^+_{E(B)}$. 
\end{proof}

\begin{prop}
\label{prop:Erhomotopy}
The spectral sequence functor $E':(\ncpx)_r\to (\spse)_r$ preserves $r$-homotopy.
\end{prop}

\begin{proof}
We start with the case $n=\infty$. Here 
an $r$-homotopy between morphisms of multicomplexes $f,g:C\to D$ is defined in~\cite[Definition 3.16]{CELW18} as
an $\infty$-morphism 
of multicomplexes $h:C\to P_r(D)$ such that $\partial^-_D\circ h=f$ and $\partial^+_D\circ h=g$.

We write $\Tc$ for the category of multicomplexes with $\infty$-morphisms.
By~\cite[Theorem 3.8]{CELW18}, we have a totalisation functor $\rm{Tot}: \Tc \to \fcpx$. 
We can refine the commutative diagram given earlier to
\[
\begin{tikzcd}
\ncpx \arrow[r, "i_n"] &\infty\text{-}\mcpx \arrow[r, "i"] &\Tc \arrow[r, "\tilde{E}"]\arrow[d, "\Tot"'] &\spse \\
&&\fcpx\arrow[ur, "E"'] &
\end{tikzcd}
\]
where $i_n$ and $i$ are inclusions of subcategories and the $\Tot$ discussed earlier can be obtained as the composite of the 
inclusions and the $\Tot$ on $\Tc$.

Thus, using Proposition~\ref{prop:Erpath}, we have a morphism of spectral sequences $E(h): E(C)\to E(P_r(D))=P(r;E(D))$. 
Since $E(\partial^-_D)=\partial^-_{E(D)}$ and $E(\partial^+_D)=\partial^+_{E(D)}$, it follows from Section~\ref{S:r-homotopy} that $E(h)$ is an $r$-homotopy between
$E(f)$ and $E(g)$.

For $n<\infty$, an $r$-path object for $n$-multicomplexes was given in~\cite[Definition 5.5]{FGLW}, giving rise to a notion of $r$-homotopy.
Let us write $P_r^n$ for the $r$-path in $n$-multicomplexes, in order to distinguish it from $P_r$, the $r$-path in multicomplexes. 
These  $r$-paths can be expressed in the form $P_r^n(C)=\Lambda_r^n \otimes C$ and $P_r(C)=\Lambda_r\otimes C$.
The two can be compared in the category of multicomplexes, since there is a natural transformation $P_r^n\to P_r$ such that
$P_r^n(C)=\Lambda_r^n \otimes C \to P_r(C)= \Lambda_r\otimes C$ is $\alpha\otimes 1_C$, where $\alpha$ is the identity in bidegrees
where this is possible and zero otherwise. 

Let  $h:C\to P_r(D)$ be an $r$-homotopy from $f$ to $g$ in $\ncpx$. Then $(\alpha\otimes 1_D)\circ h$ gives an $r$-homotopy from
$f$ to $g$ in multicomplexes. In other words, the inclusion of $n$-multicomplexes into multicomplexes preserves homotopy.
\end{proof}

\begin{rmk}
The inclusion $i_n$ of $n$-multicomplexes into multicomplexes also reflects homotopy. Indeed $i(P_r^n)$ gives another 
functorial path for multicomplexes and so gives rise to an equivalent notion of homotopy.
\end{rmk}

\end{document}